\documentclass[11pt]{article}

\usepackage[a4paper,margin=1.08in]{geometry}
\usepackage[T1]{fontenc}
\usepackage{lmodern}
\usepackage{microtype}
\usepackage{amsmath,amssymb,amsthm,mathtools}
\usepackage{enumitem}
\usepackage{xcolor}
\usepackage{hyperref}
\usepackage{aliascnt}
\usepackage[nameinlink,capitalize,noabbrev]{cleveref}
\usepackage{mathrsfs}

\hypersetup{
  colorlinks=true,
  linkcolor=blue!45!black,
  citecolor=blue!45!black,
  urlcolor=blue!55!black,
  pdftitle={v1}
}

\setlist[itemize]{leftmargin=1.6em,itemsep=0.2em,topsep=0.3em}
\setlist[enumerate]{leftmargin=1.8em,itemsep=0.2em,topsep=0.3em}

\newtheorem{theorem}{Theorem}[section]

\newaliascnt{proposition}{theorem}
\newtheorem{proposition}[proposition]{Proposition}
\aliascntresetthe{proposition}

\newaliascnt{corollary}{theorem}
\newtheorem{corollary}[corollary]{Corollary}
\aliascntresetthe{corollary}

\newaliascnt{lemma}{theorem}
\newtheorem{lemma}[lemma]{Lemma}
\aliascntresetthe{lemma}

\newaliascnt{conjecture}{theorem}

\aliascntresetthe{conjecture}

\theoremstyle{definition}

\newaliascnt{definition}{theorem}
\newtheorem{definition}[definition]{Definition}
\aliascntresetthe{definition}

\newaliascnt{example}{theorem}
\newtheorem{example}[example]{Example}
\aliascntresetthe{example}

\newaliascnt{remark}{theorem}
\newtheorem{remark}[remark]{Remark}
\aliascntresetthe{remark}

\newcommand{\cA}{\mathcal A}
\newcommand{\cI}{\mathcal I}

\newcommand{\cO}{\mathcal O}
\newcommand{\cQ}{\mathcal Q}
\newcommand{\cE}{\mathcal E}
\newcommand{\cG}{\mathcal G}
\newcommand{\cL}{\mathcal L}

\newcommand{\cP}{\mathcal P}

\newcommand{\PSH}{\operatorname{PSH}}
\newcommand{\tr}{\operatorname{tr}}
\newcommand{\adj}{\operatorname{adj}}

\newcommand{\Jump}{\operatorname{Jump}}

\newcommand{\Supp}{\operatorname{Supp}}

\title{\bfseries Finite Gram Scalarization and Further Properties of Multiplier Submodule Sheaves}
\author{Jingcao Wu}
\date{}

\begin{document}
\maketitle

\begin{abstract}
Let $(E,h)$ be a singular Hermitian vector bundle on a complex manifold $X$, and let
\[
 \cE(h)_x=\{F\in\cO(E)_x:|F|_h^2\in L^1_{\mathrm{loc},x}\}
\]
be its multiplier submodule sheaf. We introduce a finite plurisubharmonic Gram scalarization condition under which the higher-rank integrability problem reduces to finitely many scalar multiplier ideals. This reduction yields coherence and strong openness without imposing a general positivity hypothesis. When the scalar weights and the varying weight have analytic singularities, it also gives a theory of module jumping numbers, including a simultaneous-residue criterion for actual jumps. Finally, we study the induced Skoda filtration: Artin--Rees yields eventual periodicity, Tor controls whether periodicity starts at the scalar threshold, and a direct-image quotient measures the obstruction to descent.
\end{abstract}


\medskip


\section{Introduction}
Multiplier submodule sheaves are higher-rank analogues of multiplier ideal sheaves. For a singular Hermitian metric $h$ on a holomorphic vector bundle $E$, the local integrability condition defining $\cE(h)$ is matrix-valued and usually cannot be reduced to finitely many scalar multiplier ideals. The purpose of this paper is to isolate a finite reduction hypothesis under which several scalar results---coherence, strong openness, discreteness of jumping numbers, and Skoda-type periodicity---can be transferred to multiplier submodules.

The basic hypothesis is a finite plurisubharmonic Gram scalarization: after a proper modification $\pi:X'\to X$, the pull-back metric is locally uniformly comparable to a finite sum
\[
 \sum_{\nu=1}^{N}e^{-\varphi_\nu}A_\nu^*A_\nu,
\]
where the $\varphi_\nu$ are plurisubharmonic and the $A_\nu$ are holomorphic bundle morphisms. Under this hypothesis, the multiplier submodule upstairs is an intersection of inverse images of scalar multiplier ideals. A canonical-form transformation rule under modifications then gives the desired properties downstairs.

The main results are as follows.

\paragraph{Coherence and strong openness.}
It is classical that multiplier ideal sheaves are coherent \cite{Nad90} and satisfy strong openness \cite{GZ15}, whereas their higher-rank analogues are not automatic. Coherence of $\cE(h)$ is known, for example, when $(E,h)$ is Nakano semi-positive \cite{Ina22}; when $(E,h)$ is Griffiths semi-positive and the unbounded locus of $\det h$ is discrete \cite{Ina22b}; when $(E,h)$ is Griffiths semi-positive and $\det h$ has analytic singularities \cite{Zou26}; or when the base manifold has no non-trivial analytic subvarieties and $(E,h)$ is Griffiths semi-positive \cite{Yan25}. Higher-rank strong openness has also been studied under certain positivity \cite{Lem17,LXYZ26,LYZ22,Wu26}. Our first result gives a different finite-reduction criterion:  
\begin{theorem}\label{t11}
Suppose that $(E,h)$ admits a finite plurisubharmonic Gram scalarization. Then $\cE(h)$ is coherent. Moreover, for every quasi-plurisubharmonic function $\psi$ on $X$,
\[
\cE(h)=\bigcup_{\varepsilon>0}\cE(h\,e^{-\varepsilon\psi}).
\]
\end{theorem}
The full discussion is in \cref{sec:coherence}. 

\paragraph{Comparison.}
The quotient metrics discussed in \cite{Hos17} provide a natural example of finite plurisubharmonic Gram scalarization, and the coherence proved there can be viewed as a special case of \cref{t11}. Moreover, if the quotient metric is induced from a general smooth source metric, it still admits a finite plurisubharmonic Gram scalarization but need not be Griffiths semi-positive. We also obtain the following separation result.
\begin{theorem}\label{t12}
There exists a Griffiths semi-positive singular Hermitian metric on the trivial rank-two
bundle over $\Delta$ which admits no finite plurisubharmonic Gram scalarization in any
neighbourhood of the origin.
\end{theorem}

Hence finite plurisubharmonic Gram scalarization and the classic positivity of vector bundles are, in general, independent. See \cref{sec:quotient} and \cref{sec:failure}.

\paragraph{Jumping numbers.}
Assume that the scalarization weights and the varying weight have analytic singularities. Then the module jumping numbers are discrete, and they are rational when the singularity coefficients are rational. Divisorial data determine the candidate values, but an actual module jump is detected by the image of the Gram morphisms, equivalently by a simultaneous residue map. These results are developed in \cref{sec:jumping-numbers}.

\paragraph{Skoda periodicity}
Classical Skoda theory gives the eventual recursion
\[
\mathcal{I}(\mathfrak{a}^{t+1})=\mathfrak{a}\,\mathcal{I}(\mathfrak{a}^{t})
\]
and its variants; the analytic division theorem goes back to Skoda \cite{Sko72}, and the resulting periodic behaviour of jumping numbers was systematically studied in \cite{ELSV04}. We are not aware of a corresponding periodicity theorem for multiplier submodule sheaves of singular Hermitian vector bundles. In \cref{sec:skoda}, we study the integer-indexed filtration associated with a rational-coefficient ideal-type weight. On a scalarization chart $U\subset X'$, write
\[
\pi^*\psi=\frac rs\log\!\left(\sum_{\beta=1}^{q}|g_\beta|^2\right)+O(1),\qquad r,s\in\mathbb N,\quad \gcd(r,s)=1,
\]
and set $\mathfrak a=(g_1,\ldots,g_q)\subset\cO_U$. Then:
\begin{theorem}\label{t13}
Suppose that $(E,h)$ admits a finite plurisubharmonic Gram scalarization whose scalar weights have analytic singularities. For every scalarization chart $U$ as above, there exists an integer $k_1=k_1(U)$ such that
\[
\mathcal E\bigl(\pi^*(h\, e^{-(k+s)\psi})\bigr)|_U=\mathfrak a^r\,\mathcal E\bigl(\pi^*(h \,e^{-k\psi})\bigr)|_U,\qquad k\ge k_1.
\]
\end{theorem}
We also obtain a Tor criterion for periodicity to begin already at the scalar Skoda threshold. For the descent statement, suppose locally on a relatively compact neighbourhood $W\Subset X$ that
\[
\psi=\frac rs\log\!\left(\sum_{j=1}^{q}|b_j|^2\right)+O(1),\qquad \mathfrak b=(b_1,\ldots,b_q),
\]
put $X'_W=\pi^{-1}(W)$, $\pi_W=\pi|_{X'_W}$, $\mathfrak a=\mathfrak b\,\cO_{X'_W}$, and $h_k=h\, e^{-k\psi}$. After choosing a finite scalarization cover of $\pi^{-1}(\overline W)$, let $k_1$ be a common stabilization index. Then:
\begin{theorem}\label{t14}
For every integer $k\ge k_1$,
\[
\mathfrak b^r\,\mathcal E(h_k)|_W\subseteq\mathcal E(h_{k+s})|_W.
\]
Let $\mathcal F_k=K_{X'_W}\otimes\mathcal E(\pi_W^*h_k)$. Equality holds if and only if
\[
\pi_{W*}(\mathfrak a^r\,\mathcal F_k)=\mathfrak b^r\,\pi_{W*}\mathcal F_k.
\]
\end{theorem}
Thus the higher-rank problem has two additional layers beyond scalar Skoda periodicity: compatibility with the Gram image, controlled by Artin--Rees and Tor, and compatibility with direct image, measured by the descent defect sheaf.

\section{Preliminary}

Let $X$ be a complex manifold, let $E\to X$ be a holomorphic vector bundle of rank $r$, and let $h$ be a singular Hermitian metric on $E$. We set
\[
 \cE(h)_x=\bigl\{F\in\cO(E)_x: |F|_h^2\in L^1_{\mathrm{loc},x}\bigr\}.
\]
For a plurisubharmonic function $\varphi$, its multiplier ideal is
\[
 \cI(\varphi)_x=\bigl\{f\in\cO_{X,x}: |f|^2e^{-\varphi}\in L^1_{\mathrm{loc},x}\bigr\}.
\]
If $H_1$ and $H_2$ are non-negative Hermitian forms, we write $H_1\asymp H_2$ locally if there is a locally uniform constant $C\geq1$ such that
\[
 C^{-1}H_2\leq H_1\leq CH_2.
\]
Comparable metrics determine the same multiplier submodule sheaf.

The following version of strong openness \cite{GZ15} and a canonical transformation rule are applied repeatedly in the text.
 
\begin{lemma}\label{lem:mixed-openness}
Let $\varphi$ and $\psi$ be plurisubharmonic functions on a neighbourhood of $x\in X$. Then there exists $\varepsilon>0$ such that
\[
\cI(\varphi)_{x}=\cI(\varphi+\varepsilon\psi)_{x}.
\]
\end{lemma}

\begin{lemma}\label{prop:birational-covariance}
Let $\pi:X'\to X$ be a proper modification between complex manifolds. For every singular
Hermitian metric $h$ on $E$, every quasi-plurisubharmonic function $\psi$ on $X$, and every $t\geq0$, there is a natural equality of subsheaves of $K_X\otimes E$:
\begin{equation}\label{eq:birational-covariance}
 \pi_*\bigl(K_{X'}\otimes\cE(\pi^*h\,e^{-t\pi^*\psi})\bigr)=K_X\otimes\cE(h\, e^{-t\psi}).
\end{equation}
\end{lemma}

\begin{proof}
By the projection formula and the standard identity $\pi_*K_{X'}=K_X$ for a modification
between manifolds, a local section of the left-hand side is naturally identified with an
$E$-valued canonical form on $X$. It remains to compare its local integrability on the two spaces.

Choose local coordinates $z=(z_1,\ldots,z_n)$ on $X$ and $w=(w_1,\ldots,w_n)$ on $X'$, and write
\[
 J_\pi(w)=\det\left(\frac{\partial z_i}{\partial w_j}\right).
\]
If
\[
 \sigma=F(z)\,dz_1\wedge\cdots\wedge dz_n,
\]
then on the locus where $\pi$ is biholomorphic,
\[
 \pi^*\sigma=F(\pi(w))J_\pi(w)\,dw_1\wedge\cdots\wedge dw_n.
\]
Consequently,
\[
 |\pi^*\sigma|_{\pi^*h}^2e^{-t\pi^*\psi}\,d\lambda_w=|F(\pi(w))|_h^2e^{-t\psi(\pi(w))}|J_\pi(w)|^2\,d\lambda_w.
\]
The change-of-variables formula identifies the corresponding integrals away from the
exceptional and critical loci. These are analytic sets of Lebesgue measure zero, so they do not affect local integrability. Thus $\sigma$ is locally square-integrable with respect to $h\, e^{-t\psi}$ if and only if its pull-back is locally square-integrable with respect to $\pi^*h\, e^{-t\pi^*\psi}$, proving \eqref{eq:birational-covariance}.
\end{proof}

\section{Finite plurisubharmonic Gram scalarizations}\label{sec:Gram}

\begin{definition}\label{d31}
The pair $(E,h)$ is said to admit a \emph{finite plurisubharmonic Gram scalarization} if
there is a proper modification
\[
 \pi:X'\longrightarrow X
\]
such that, on every sufficiently small coordinate ball $U\Subset X'$, there are
plurisubharmonic functions $\varphi_1,\ldots,\varphi_N$, holomorphic bundle morphisms
\[
 A_\nu:\pi^*E|_U\longrightarrow\cO_U^{m_\nu},\qquad 1\leq\nu\leq N,
\]
and a constant $C\geq1$ satisfying
\begin{equation}\label{eq:gram-comparison}
C^{-1}\sum_{\nu=1}^N e^{-\varphi_\nu}|A_\nu F|^2\le |F|_{\pi^{\ast}h}^2\le C \sum_{\nu=1}^N e^{-\varphi_\nu}|A_\nu F|^2.
\end{equation}
for every local holomorphic section $F$ of $\pi^*E$. The scalarization is \emph{exact} if either of the equality in \eqref{eq:gram-comparison} holds.
\end{definition}

\subsection{Coherence and strong openness}\label{sec:coherence}
\begin{theorem}[=\cref{t11}]\label{thm:finite-gram}
Suppose that $(E,h)$ admits a finite plurisubharmonic Gram scalarization. On a scalarization chart $U\subset X'$ one has
\begin{equation}\label{eq:gram-module-formula}
\cE(\pi^{\ast}h)|_U=\bigcap_{\nu=1}^N A_\nu^{-1}\left(\cI(\varphi_\nu)\cO_U^{m_\nu}\right).
\end{equation}
In particular, $\cE(h)$ is coherent. Moreover, for every quasi-plurisubharmonic function
$\psi$ on $X$,
\begin{equation}\label{eq:gram-strong-open}
\cE(h)=\bigcup_{\varepsilon>0}\cE(h\,e^{-\varepsilon\psi}).
\end{equation}
\end{theorem}

\begin{proof}
Write
\[
A_\nu F=(g_{\nu,1},\ldots,g_{\nu,m_\nu})^T.
\]
By \eqref{eq:gram-comparison}, the condition $F\in\cE(\pi^{\ast}h)$ is equivalent to
\[
\sum_{\nu=1}^N\sum_{j=1}^{m_\nu}|g_{\nu,j}|^2e^{-\varphi_\nu}\in L^1_{\mathrm{loc}}.
\]
Since the sum is finite and all terms are non-negative, this is equivalent to
$g_{\nu,j}\in\cI(\varphi_\nu)$ for every $\nu,j$. This proves \eqref{eq:gram-module-formula}.

Set
\[
\cQ_\nu=\cO_U^{m_\nu}/\cI(\varphi_\nu)\cO_U^{m_\nu}.
\]
The morphism
\[
\Theta:\cO(\pi^{\ast}E)|_U\longrightarrow\bigoplus_{\nu=1}^N\cQ_\nu,\qquad F\longmapsto([A_1F],\ldots,[A_NF]).
\]
has kernel $\cE(\pi^*h)|_U$. Hence $\cE(\pi^*h)$ is coherent. Therefore $K_{X'}\otimes\cE(\pi^*h)$ is coherent, and so is its proper direct image. By \cref{prop:birational-covariance}, that direct image equals $K_X\otimes\cE(h)$. Since $K_X$ is invertible, $\cE(h)$ is coherent.

For strong openness, after adding a constant to $\psi$ we may assume locally that $\psi\le 0$; then the inclusion
\[
\bigcup_{\varepsilon>0}\mathcal E(h\,e^{-\varepsilon\psi})\subset\mathcal E(h)
\]
is immediate. Fix $x\in X$ and $F\in\mathcal E(h)_x$, and write $\sigma=F\,dz_1\wedge\cdots\wedge dz_n$. By \cref{prop:birational-covariance}, $\pi^*\sigma$ is locally square-integrable with respect to $\pi^*h$ near the compact fibre $\pi^{-1}(x)$.

For each $y\in\pi^{-1}(x)$, choose a scalarization chart $V_y$ and trivialize $K_{X'}$. If $F'$ is the local coefficient of $\pi^*\sigma$, then \eqref{eq:gram-module-formula} gives that every component of $A_\nu F'$ lies in $\mathcal I(\varphi_\nu)$. Locally write
$\pi^*\psi=u+\rho$, with $u$ plurisubharmonic and $\rho$ smooth. Applying \cref{lem:mixed-openness} to the finitely many components, and shrinking $V_y$ if necessary, gives $\varepsilon_y>0$ such that
\[
F'\in\mathcal E(\pi^*h\, e^{-\varepsilon_y\pi^*\psi})\quad\text{on }V_y.
\]
A finite subcover of $\pi^{-1}(x)$ and $\varepsilon=\min\varepsilon_y$ give, after shrinking a neighbourhood $W$ of $x$, one $\varepsilon>0$ valid on $\pi^{-1}(W)$. \cref{prop:birational-covariance} then yields $F\in\mathcal E(h\,e^{-\varepsilon\psi})_x$.
\end{proof}

\subsection{Quotient metric}\label{sec:quotient}

The following quotient-metric example is due to \cite{Hos17}; we include the details to make the scalarization explicit. Let $U\subset\mathbb C^n$ be a domain and let
\[
 A:\cO_U^N\longrightarrow\cO_U^r
\]
be a generically surjective holomorphic morphism. Write its columns as $A=(a_1,\ldots,a_N)$ and let
\[
 U^\circ=\{x\in U:\operatorname{rank}A(x)=r\}.
\]
Equip $\mathcal O_U^N$ with the Euclidean metric and let $h$ be the induced quotient metric on $\mathcal O_U^r$. On $U^\circ$, $h=(AA^*)^{-1}$; this defines a singular Hermitian metric on $U$. The following theorem was originally proved in \cite{Hos17} and the coherence may be viewed as a direct consequence of \cref{thm:finite-gram}.

\begin{theorem}\label{t33}
Define the plurisubharmonic function
\[
 \Phi=\log\left(\sum_{\substack{I\subset\{1,\ldots,N\}\\|I|=r}}|\Delta_I(A)|^2\right).
\]
For every subset $J=\{j_1,\ldots,j_{r-1}\}$ of cardinality $r-1$, define
\[
 L_J:\cO_U^r\longrightarrow\cO_U,\qquad L_J(F)=\det(a_{j_1},\ldots,a_{j_{r-1}},F).
\]
Then on $U$,
\begin{equation}\label{eq:quotient-scalarization}
 |F|_h^2=e^{-\Phi}\sum_{\substack{J\subset\{1,\ldots,N\}\\|J|=r-1}}|L_J(F)|^2.
\end{equation}
Thus $h$ admits an exact finite plurisubharmonic Gram scalarization and $\cE(h)$ is
coherent.
\end{theorem}

\begin{proof}
For $F\in\mathbb C^r$ and $x\in U^\circ$, the minimizing vector in the definition of the
quotient norm is
\[
 x_{\min}=A^*(AA^*)^{-1}F.
\]
Hence
\begin{equation}\label{eq:quotient-inverse}
 |F|_h^2=F^*(AA^*)^{-1}F.
\end{equation}
By Cauchy--Binet,
\begin{equation}\label{eq:maximal-minors}
 \det(AA^*)=\sum_{\substack{I\subset\{1,\ldots,N\}\\|I|=r}}|\Delta_I(A)|^2=e^\Phi.
\end{equation}
We claim that
\begin{equation}\label{eq:adjugate-minors}
 F^*\adj(AA^*)F=\sum_{\substack{J\subset\{1,\ldots,N\}\\|J|=r-1}}|L_J(F)|^2.
\end{equation}
Both sides are invariant under unitary changes of coordinates in $\mathbb C^r$. We may
therefore assume $F=(|F|,0,\ldots,0)^T$. Let $A'$ be obtained from $A$ by deleting its
first row. The upper-left entry of $\adj(AA^*)$ equals $\det(A'A'^*)$. Applying Cauchy--Binet again gives
\[
 \det(A'A'^*)=\sum_{|J|=r-1}|\det A'_J|^2.
\]
In the chosen coordinates, $L_J(F)=\pm|F|\det A'_J$, proving \eqref{eq:adjugate-minors}. Combining \eqref{eq:quotient-inverse}--\eqref{eq:adjugate-minors} with $(AA^*)^{-1}=\adj(AA^*)/\det(AA^*)$ proves \eqref{eq:quotient-scalarization} on $U^\circ$. Since the complement of $U^\circ$ is analytic and has measure zero, the almost-everywhere identity is sufficient for the singular metric and its multiplier submodule. Coherence follows from \cref{thm:finite-gram}.
\end{proof}

\begin{corollary}\label{c41}
With the notation above,
\[
\cE(h)=\bigcap_{\substack{J\subset\{1,\ldots,N\}\\ |J|=r-1}}L_J^{-1}\bigl(\cI(\Phi)\bigr).
\]
For every quasi-plurisubharmonic function $\psi$ on $U$,
\[
 \cE(h)=\bigcup_{\varepsilon>0}\cE(h\, e^{-\varepsilon\psi}).
\]
\end{corollary}

\begin{proof}
The first assertion follows directly from \eqref{eq:quotient-scalarization}, because the sum has finitely many non-negative terms. The second follows from \cref{thm:finite-gram}. 
\end{proof}

More generally, let
\[
q:(V,g)\longrightarrow E
\]
be a generically surjective morphism of holomorphic bundles, where $g$ is a smooth Hermitian metric. In local frames, if $G$ is the matrix of $g$ and $A$ is the matrix of $q$, then the quotient metric is
\[
 h_{q,g}=(AG^{-1}A^*)^{-1}.
\]
On every relatively compact coordinate neighbourhood there exist constants $0<c\leq C$ such that $cI\leq G^{-1}\leq CI$, and hence
\[
 C^{-1}(AA^*)^{-1}\leq h_{q,g}\leq c^{-1}(AA^*)^{-1}
\]
on the full-rank locus. Therefore
\[
 |F|_{h_{q,g}}^2\asymp e^{-\Phi}\sum_{|J|=r-1}|L_J(F)|^2.
\]
Thus quotient metrics of smoothly metrized bundles also admit finite plurisubharmonic Gram scalarizations, although in general only up to local uniform equivalence.

\begin{remark}
If the source metric is singular, $G^{-1}$ need not be locally comparable with a smooth
positive-definite matrix.  Thus the preceding reduction requires additional hypotheses.
\end{remark}

\subsection{Failure of finite plurisubharmonic Gram scalarization}\label{sec:failure}

Finite plurisubharmonic Gram scalarizability does not imply Griffiths semi-positivity. We now show that the converse also fails in general. 

For a positive Hermitian matrix $H$, write
\[
 \widehat H=\frac{H}{\tr H}.
\]

\begin{proposition}\label{p36}
Let $z$ be the coordinate on a disk $\Delta_\varepsilon\subset\mathbb C$ centred at the origin, and let
\[
G(z)=\sum_{\nu=1}^{N}\rho_\nu(z)A_\nu(z)^*A_\nu(z),
\]
where $\rho_\nu(z)>0$ and the $A_\nu$ are holomorphic matrices on $\Delta_\varepsilon$. If $z_j\to0$ and $\widehat G(z_j)\to P$, where $P$ is a rank-one projection, then $P$ belongs to a finite set determined by the matrices $A_\nu$.
\end{proposition}

\begin{proof}
Let $m_\nu$ be the minimal vanishing order at $0$ among the entries of $A_\nu$. Then
\[
 A_\nu(z)=z^{m_\nu}(B_\nu+O(z)),\qquad B_\nu\neq0,
\]
and
\[
 A_\nu(z)^*A_\nu(z)=|z|^{2m_\nu}(C_\nu+o(1))\qquad\text{with } C_\nu=B_\nu^*B_\nu.
\]
Set
\[
 t_{\nu,j}=\rho_\nu(z_j)|z_j|^{2m_\nu}\tr C_\nu.
\]
After passing to a subsequence,
\[
 \frac{t_{\nu,j}}{\sum_\mu t_{\mu,j}}\longrightarrow\alpha_\nu\qquad\text{with } \alpha_\nu\geq0\qquad\text{and}\qquad\sum_\nu\alpha_\nu=1.
\]
The normalized limit is therefore
\[
 P=\sum_\nu\alpha_\nu\frac{C_\nu}{\tr C_\nu}.
\]
If $v\in\ker P$, then
\[
 0=\langle Pv,v\rangle=\sum_\nu\alpha_\nu\frac{\langle C_\nu v,v\rangle}{\tr C_\nu}.
\]
All summands are non-negative, so every $C_\nu$ with $\alpha_\nu>0$ vanishes on
$\ker P$. Its range is therefore contained in the one-dimensional space $\operatorname{im}\,P$, and $C_\nu/\tr C_\nu=P$. Hence $P$ is one of finitely many normalized matrices $C_\nu/\tr C_\nu$.
\end{proof}

\begin{lemma}\label{lem:comparable-limit}
Suppose $H_j,G_j$ are positive Hermitian matrices and $C^{-1}G_j\leq H_j\leq CG_j$ for a fixed $C\geq1$. If $\widehat H_j\to P$ for a rank-one projection $P$, then $\widehat G_j\to P$.
\end{lemma}

\begin{proof}
The trace inequalities give
$C^{-1}\tr G_j\leq\tr H_j\leq C\tr G_j$. For $v\in\ker P$,
\[
0\leq\frac{\langle G_jv,v\rangle}{\tr G_j}\leq C^2\frac{\langle H_jv,v\rangle}{\tr H_j}\longrightarrow0.
\]
Every cluster point of $\widehat G_j$ therefore has kernel containing $\ker P$; since it is non-negative of trace one, it must equal $P$.
\end{proof}

Now let $\Delta=\{z\in\mathbb C:|z|<1\}$ and $E=\Delta\times\mathbb C^2$. Choose $\theta\in\mathbb R$ with $\theta/\pi\notin\mathbb Q$, and set
\[
u_n=\begin{pmatrix}\cos(n\theta)\\ \sin(n\theta)\end{pmatrix},\qquad P_n=u_nu_n^*.
\]
Define a metric on $E^*$ by
\begin{equation}\label{eq:k-counterexample}
 k(z)=|z|^2I_2+\sum_{n=1}^{\infty}e^{-4^n}|z|^{2^{1-n}}P_n.
\end{equation}
The series converges uniformly on $\Delta$. It is positive definite for $z\neq0$; define
$h=k^{-1}$ almost everywhere.

\begin{proposition}\label{prop:counterexample-positive}
The metric $k$ is Griffiths semi-negative; consequently, $h$ is Griffiths semi-positive.
\end{proposition}

\begin{proof}
For a holomorphic section $s=(s_1,s_2)$ of $E^*$,
\[
|s|_k^2=|z|^2(|s_1|^2+|s_2|^2)+\sum_{n=1}^{\infty}e^{-4^n}|z|^{2^{1-n}}|\langle u_n,s\rangle|^2.
\]
Each summand is the exponential of a subharmonic function (with the value $0$ allowed at
its zero set). Hence the logarithm of every finite partial sum is subharmonic. These
logarithms increase to $\log|s|_k^2$. The underlying sums converge locally uniformly, so
the limit is upper semi-continuous and locally bounded above away from the common zero set. The monotone convergence theorem for subharmonic functions therefore gives $\log|s|_k^2\in\PSH(\Delta)$. This is precisely Griffiths semi-negativity of $k$.
\end{proof}

Put $x=-\log |z|^2$. The coefficient of $P_n$ in \eqref{eq:k-counterexample} is
\[
 w_n(x)=\exp(-4^n-2^{-n}x).
\]
Choose $x_n=3\cdot2^{3n}$ and $z_n$ with $-\log|z_n|^2=x_n$. Then
\[
 w_n(x_n)=e^{-4^{n+1}},
\]
and a direct comparison of exponents gives
\[
 \frac{|z_n|^2+\sum_{m\neq n}w_m(x_n)}{w_n(x_n)}\longrightarrow0.
\]
Thus
\[
 \widehat k(z_n)-P_n\longrightarrow0.
\]
Because $\theta/\pi$ is irrational, the lines $[\mathbb C u_n]$ are dense in the real locus $\mathbb{RP}^1\subset\mathbb{CP}^1$. Hence the normalized matrices $\widehat k(z_n)$ have infinitely many rank-one cluster points. For a $2\times2$ positive matrix,
\[
 \frac{k^{-1}}{\tr(k^{-1})}=\frac{\adj k}{\tr(\adj k)}.
\]
If $P$ is a rank-one projection, then $\adj P=I-P$. Therefore the normalized matrices
$\widehat h(z_n)$ also have infinitely many rank-one cluster points.

\begin{theorem}[=\cref{t12}]\label{thm:no-finite-scalarization}
There exists a Griffiths semi-positive singular Hermitian metric on the trivial rank-two
bundle over $\Delta$ which admits no finite plurisubharmonic Gram scalarization in any
neighbourhood of the origin.
\end{theorem}

\begin{proof}
For the metric $h$ constructed above, $\widehat h(z_n)$ has infinitely many rank-one cluster points. Since every proper modification of a non-singular complex curve is an isomorphism, a finite scalarization in the sense of \cref{d31} would already give a local scalarization on the disk. If $h$ admitted such a scalarization, its Gram sum $G=\sum e^{-\varphi_\nu}A_\nu^*A_\nu$ would be uniformly comparable with $h$. By \cref{lem:comparable-limit}, every rank-one cluster point of $\widehat h(z_n)$ would be the corresponding cluster point of $\widehat G(z_n)$. This contradicts \cref{p36}, which permits only finitely many such limits.
\end{proof}

\section{Jumping numbers}
\label{sec:jumping-numbers}

In this section, we study the jumping behaviour of multiplier submodule sheaves arising from a finite plurisubharmonic Gram scalarization. The resulting module jumping numbers are controlled by finitely many scalar multiplier ideals.

\subsection{Setup and scalar upper bounds}

Suppose that $(E,h)$ admits a finite plurisubharmonic Gram scalarization. Let $\psi$ be a quasi-plurisubharmonic function on $X$. After adding a constant, we may assume locally that $\psi\leq0$. For $t\geq0$, set $h_t=h\,e^{-t\psi}$. Then $\cE(h_t)$ is decreasing in $t$. On a scalarization chart $U\subset X'$, one has 
\begin{equation}
\label{eq:module-family}
\mathcal E(\pi^{\ast}h_{t})|_{U}=\bigcap_{\nu=1}^{N}A_{\nu}^{-1}\left(\mathcal I(\varphi_{\nu}+t\pi^{\ast}\psi)\mathcal O_{U}^{m_{\nu}}\right).
\end{equation}

\begin{definition}
\label{def:module-jumping-number}
Fix $x\in X$. A real number $c>0$ is a \emph{jumping number} of the family $\cE(h_t)$ at $x$ if
\[
 \cE(h_c)_x\subsetneq\cE(h_{c-\varepsilon})_x
\]
for every sufficiently small $\varepsilon>0$. The set of such numbers is denoted $\Jump_x(h;\psi)$.
\end{definition}

For $x'\in X'$ and a scalarization weight $\varphi_\nu$ on the birational model $X'$, denote by $\Jump_{x'}(\varphi_\nu;\pi^*\psi)$ the jumping set of $t\mapsto\cI(\varphi_\nu+t\pi^*\psi)_{x'}$.

\begin{proposition}
\label{prop:jumping-contained-scalar}
For every $x\in X$,
\[
\Jump_x(h;\psi)\subseteq\bigcup_{x'\in\pi^{-1}(x)}\ \bigcup_\nu\Jump_{x'}(\varphi_\nu;\pi^*\psi),
\]
where the right-hand side is formed using a finite collection of scalarization charts covering $\pi^{-1}(x)$ and the scalar weights occurring on those charts.
\end{proposition}

\begin{proof}
Suppose that $c$ does not belong to the union on the right. At each point of the compact fibre $\pi^{-1}(x)$, the scalar multiplier ideals at $c$ and $c-\varepsilon$ agree for all sufficiently small $\varepsilon>0$. Since the ideals are coherent, this equality persists on a neighbourhood of that point. A finite subcover of the fibre and the minimum of the corresponding constants yield $\varepsilon_0>0$ such that
\[
\cI(\varphi_\nu+c\pi^*\psi)=\cI(\varphi_\nu+(c-\varepsilon)\pi^*\psi)
\]
on all charts near the fibre whenever $0<\varepsilon<\varepsilon_0$. \eqref{eq:module-family} gives equality of the multiplier submodules upstairs, and \cref{prop:birational-covariance} gives $\cE(h_c)_x=\cE(h_{c-\varepsilon})_x$.
\end{proof}

The inclusion can be strict because the scalarization is not intrinsic: a scalar jump in $\Jump_{x'}(\varphi_\nu;\pi^*\psi)$ may be invisible on the simultaneous image of the Gram morphisms.

\subsection{Analytic singularities and discreteness}

For the remainder of the paper, assume that the scalar weights and $\psi$ have analytic singularities. On a fixed scalarization chart $U$, write
\[
\varphi_\nu=c_\nu\log\left(\sum_{\alpha=1}^{q_\nu}|f_{\nu,\alpha}|^2\right)+O(1),\qquad
\pi^*\psi=d\log\left(\sum_{\beta=1}^{q}|g_\beta|^2\right)+O(1),
\]
where $c_\nu,d>0$. Let $\mathfrak a_\nu=(f_{\nu,1},\ldots,f_{\nu,q_\nu})$ and $\mathfrak a=(g_1,\ldots,g_q)$. Choose a common log resolution $\mu:Y\to U$ and write
\[
\mathfrak a_\nu\,\cO_Y=\cO_Y(-F_\nu),\qquad\mathfrak a\,\cO_Y=\cO_Y(-G),
\]
with
\[
F_\nu=\sum_i m_{\nu,i}D_i,\qquad G=\sum_i n_iD_i\qquad\text{and}\qquad K_{Y/U}=\sum_i a_iD_i.
\]
Then 
\[
\mathcal I(\varphi_{\nu}+t\pi^{\ast}\psi)=\mu_{*}\mathcal O_{Y}\left(K_{Y/U}-\left\lfloor c_{\nu}F_{\nu}+tdG\right\rfloor\right).
\]
For each $\nu$, the possible scalar jumping values are contained in
\begin{equation}
\label{eq:candidate-jumping-values}
\bigcup_{i:\,n_i>0}\left\{\frac{k-c_\nu m_{\nu,i}}{d n_i}:\ k\in\mathbb Z\right\}\cap\mathbb R_{>0}.
\end{equation}

\begin{theorem}
\label{thm:discreteness-jumping-module}
If all scalarization weights and $\psi$ have analytic singularities, then $\Jump_x(h;\psi)$ is a discrete subset of $\mathbb R_{>0}$ for every $x\in X$. In particular, it has no finite accumulation point.
\end{theorem}

\begin{proof}
Choose finitely many scalarization charts covering the compact fibre $\pi^{-1}(x)$. After shrinking the charts, choose a log resolution of the finitely many ideals appearing on each one. Each resolution has only finitely many relevant divisors $D_i$, and \eqref{eq:candidate-jumping-values} is locally finite in $\mathbb R_{>0}$. The finite union of these candidate sets is locally finite.  The conclusion follows from \cref{prop:jumping-contained-scalar}. 
\end{proof}

\begin{corollary}
\label{cor:rationality-jumping-module}
If, in addition, all coefficients $c_\nu$ and $d$ are rational, then
\[
 \Jump_x(h;\psi)\subset\mathbb Q_{>0}.
\]
\end{corollary}

\subsection{Actual jumps and simultaneous residues}

Fix one scalarization chart $U$ and set
\[
 \widetilde{\mathcal M}(t)=\cE(\pi^*h_t)|_U.
\]
For a candidate value $c$, choose $0<\varepsilon\ll1$ so that there is no other candidate
value in $[c-\varepsilon,c)$. Put
\[
 \mathcal J_\nu^-=\cI(\varphi_\nu+(c-\varepsilon)\pi^*\psi)\qquad\text{and}\qquad\mathcal J_\nu^+=\cI(\varphi_\nu+c\pi^*\psi).
\]
Write
\[
 \alpha_{\nu,i}=c_\nu m_{\nu,i},\qquad \beta_i=dn_i,
\]
and define the active set and reduced jumping divisor by
\[
S_\nu(c)=\{i:\beta_i>0,\ \alpha_{\nu,i}+c\beta_i\in\mathbb Z\},\qquad D_\nu(c)=\sum_{i\in S_\nu(c)}D_i.
\] 
For sufficiently small $\varepsilon$,
\[
\lfloor c_\nu F_\nu+cdG\rfloor=\lfloor c_\nu F_\nu+(c-\varepsilon)dG\rfloor+D_\nu(c).
\]
Let
\[
L_\nu(c)=\cO_Y\bigl(K_{Y/U}-\lfloor c_\nu F_\nu+cdG\rfloor\bigr).
\]
Then
\[
\mathcal J_\nu^+=\mu_*L_\nu(c)\qquad\text{and}\qquad\mathcal J_\nu^-=\mu_*\bigl(L_\nu(c)(D_\nu(c))\bigr).
\]
The morphism
\begin{equation}\label{eq:rho-c}
\rho_c:\widetilde{\mathcal M}(c-\varepsilon)\longrightarrow\bigoplus_{\nu=1}^{N}(\mathcal J_\nu^-/\mathcal J_\nu^+)^{\oplus m_\nu},\qquad F\longmapsto([A_1F],\ldots,[A_NF]),
\end{equation}
has kernel $\widetilde{\mathcal M}(c)$. Hence
\begin{equation}\label{eq:upstairs-jump-image}
 \frac{\widetilde{\mathcal M}(c-\varepsilon)}{\widetilde{\mathcal M}(c)}\simeq\operatorname{im}\,\rho_c.
\end{equation}

For a vector-valued germ $g=(g_1,\ldots,g_m)$, put
\[
 v_{D_i}(g)=\min_j\operatorname{ord}_{D_i}(\mu^*g_j),
\]
and define
\[
q_{\nu,i}(c)=\lfloor\alpha_{\nu,i}+c\beta_i\rfloor-a_i,\qquad\chi_{\nu,i}(c)=
 \begin{cases}
 1,&i\in S_\nu(c),\\
 0,&i\notin S_\nu(c).
 \end{cases}
\]
Then $F\in\widetilde{\mathcal M}(c-\varepsilon)$ if and only if
\begin{equation}\label{eq:prejump-valuations}
v_{D_i}(A_\nu F)\geq q_{\nu,i}(c)-\chi_{\nu,i}(c)\quad\text{for all }\nu,i,
\end{equation}
and $F\in\widetilde{\mathcal M}(c)$ if and only if $v_{D_i}(A_\nu F)\geq q_{\nu,i}(c)$ for all $\nu,i$.

\begin{theorem}[Upstairs jump criterion]\label{thm:upstairs-jump}
For a candidate value $c$, the following are equivalent:
\begin{enumerate}[label=\textup{(\arabic*)},leftmargin=2.4em]
 \item $\widetilde{\mathcal M}(c-\varepsilon)/\widetilde{\mathcal M}(c)\neq0$;
 \item $\rho_c\neq0$;
 \item there is a germ $F$ satisfying \eqref{eq:prejump-valuations} such that, for some
       active pair $(\nu,i)$,
       \[
        v_{D_i}(A_\nu F)=q_{\nu,i}(c)-1;
       \]
 \item the simultaneous residue of some pre-jump admissible section is non-zero.
\end{enumerate}
Moreover,
\[
\Supp\frac{\widetilde{\mathcal M}(c-\varepsilon)}{\widetilde{\mathcal M}(c)}\subseteq\bigcup_{\nu=1}^{N}\mu(D_\nu(c)).
\]
\end{theorem}

\begin{proof}
The equivalence of (1) and (2) is \eqref{eq:upstairs-jump-image}. The valuation description of the two modules shows that their quotient is non-zero precisely when one of the newly imposed active inequalities is sharp, proving the equivalence with (3).

The exact sequence
\[
0\to L_\nu(c)\to L_\nu(c)(D_\nu(c))\to L_\nu(c)(D_\nu(c))|_{D_\nu(c)}\to0
\]
and local vanishing for multiplier ideals give
\[
\mathcal J_\nu^-/\mathcal J_\nu^+\simeq\mu_*\bigl(L_\nu(c)(D_\nu(c))|_{D_\nu(c)}\bigr).
\]
Thus \eqref{eq:rho-c} is the simultaneous residue map, proving (4). The support statement
follows from the same description.
\end{proof}

\begin{example}
Let $U=\Delta$, $E=\cO_\Delta$, $A(f)=zf$, $\varphi=0$, and $\psi=\log|z|^2$. Then
\[
 \cI(t\psi)=(z^{\lfloor t\rfloor}).
\]
At $c=1$, one has $\mathcal J^-=\cO_\Delta$ and $\mathcal J^+=(z)$, but
\[
A^{-1}(\mathcal J^-)=\cO_\Delta\qquad\text{and}\qquad A^{-1}(\mathcal J^+)=\{f:zf\in(z)\}=\cO_\Delta.
\]
Thus the scalar ideal jumps while the module does not. At $c=2$, $\mathcal J^-=(z)$ and $\mathcal J^+=(z^2)$, so
\[
A^{-1}(\mathcal J^-)=\cO_\Delta\qquad\text{and}\qquad A^{-1}(\mathcal J^+)=(z),
\]
and $2$ is an actual module jumping number. The divisorial coefficients determine the
candidate values, whereas the Gram morphisms and their simultaneous syzygies determine
which candidates survive.
\end{example}

An upstairs jump may disappear after descent.

\begin{proposition}[Descent of jump quotients]\label{prop:jump-descent}
Let 
\[
\cQ_c=\cE(h_{c-\varepsilon})/\cE(h_{c})\qquad\textrm{and}\qquad\widetilde \cQ_c=\cE(\pi^*h_{c-\varepsilon})/\cE(\pi^*h_c).
\]
Then
\begin{equation}\label{eq:jump-pushforward-image}
K_X\otimes \cQ_c\simeq\mathrm{im}\,\left[\pi_*\bigl(K_{X'}\otimes\cE(\pi^*h_{c-\varepsilon})\bigr)\longrightarrow\pi_*\bigl(K_{X'}\otimes\widetilde \cQ_c\bigr)\right].
\end{equation}
Consequently, a downstairs jump implies an upstairs jump over the fibre, but the converse
need not hold. A candidate is an actual downstairs jump exactly when the canonical
push-forward of the simultaneous residue map has non-zero image.
\end{proposition}

\begin{proof}
Apply $\pi_*$ to
\[
0\to K_{X'}\otimes\cE(\pi^*h_c)\to K_{X'}\otimes\cE(\pi^*h_{c-\varepsilon})\to K_{X'}\otimes\widetilde \cQ_c\to0.
\]
Left exactness identifies the quotient of the first two direct images with the image in the third. The transformation formula \eqref{eq:birational-covariance} identifies that quotient with $K_X\otimes \cQ_c$, giving \eqref{eq:jump-pushforward-image}.
\end{proof}

\section{Skoda periodicity}\label{sec:skoda}

Fix an $n$-dimensional scalarization chart $U\subset X'$. Write
\[
\varphi_\nu=c_\nu\log\!\left(\sum_{\alpha=1}^{q_\nu}|f_{\nu,\alpha}|^2\right)+O(1),\qquad
\pi^*\psi=d\log\!\left(\sum_{\beta=1}^{q}|g_\beta|^2\right)+O(1),
\]
where $c_\nu>0$ and
\[
d=\frac{r}{s}\in\mathbb Q_{>0},\qquad r,s\in\mathbb N,\qquad\gcd(r,s)=1.
\]
Set
\[
|g|^2=\sum_{\beta=1}^{q}|g_\beta|^2,\qquad\mathfrak a=(g_1,\ldots,g_q)\subset\mathcal O_U.
\]
For each integer $k\ge0$, set
\[
h_k=h\,e^{-k\psi},\qquad\mathcal J_{\nu,k}=\mathcal I(\varphi_\nu+k\pi^*\psi)\mathcal O_U^{m_\nu}\qquad\text{and}\qquad\mathcal J_k=\bigoplus_{\nu=1}^N\mathcal J_{\nu,k}.
\]
Since locally bounded perturbations do not affect multiplier ideals, we suppress the $O(1)$-terms in the analytic-singularity expressions whenever convenient.

\begin{proposition}[Scalar Skoda periodicity]\label{prop:scalar-skoda-rational}
Put
\[
 p=\min\{n,q-1\}.
\]
For every real number $t\ge p$ and every $\nu$,
\begin{equation}\label{eq:scalar-skoda-real}
\mathcal I\!\left(\varphi_\nu+(t+1)\log|g|^2\right)=\mathfrak a\,\mathcal I\!\left(\varphi_\nu+t\log|g|^2\right).
\end{equation}
Consequently, if
\[
k_0=\left\lceil\frac{p}{d}\right\rceil=\left\lceil\frac{ps}{r}\right\rceil,
\]
then
\begin{equation}\label{eq:scalar-skoda-rational}
\mathcal J_{\nu,k+s}=\mathfrak a^r\,\mathcal J_{\nu,k}\qquad \text{for every $\nu$ and every integer $k\ge k_0$}.
\end{equation}
Hence
\begin{equation}\label{eq:J-rational}
\mathcal J_{k+s}=\mathfrak a^r\,\mathcal J_k,\qquad k\ge k_0.
\end{equation}
\end{proposition}

\begin{proof}
After shrinking $U$ and, if necessary, multiplying all generators $g_j$ by the same non-zero constant, we may assume that $|g|^2\le1$ on $U$.

We first prove~\eqref{eq:scalar-skoda-real}. The inclusion
\[
\mathfrak a\,\mathcal I\!\left(\varphi_\nu+t\log|g|^2\right)\subset\mathcal I\!\left(\varphi_\nu+(t+1)\log|g|^2\right)
\]
follows immediately from the Cauchy--Schwarz inequality. Indeed, if
\[
h_1,\ldots,h_q\in\mathcal I\!\left(\varphi_\nu+t\log|g|^2\right)
\]
and $f=\sum_{j=1}^q g_jh_j$, then
\[
 |f|^2\le |g|^2\sum_{j=1}^q|h_j|^2,
\]
and therefore
\[
|f|^2e^{-\varphi_\nu}|g|^{-2(t+1)}\le\left(\sum_{j=1}^q|h_j|^2\right)e^{-\varphi_\nu}|g|^{-2t}.
\]

For the reverse inclusion, assume first that $q\ge2$, so $p>0$, and take
\[
 f\in\mathcal I\!\left(\varphi_\nu+(t+1)\log|g|^2\right),\qquad t\ge p.
\]
By \cref{lem:mixed-openness}, there is $\varepsilon>0$ such that
\[
 f\in\mathcal I\!\left(\varphi_\nu+(t+1+\varepsilon)\log|g|^2\right).
\]
After shrinking $U$ around the point under consideration,
\begin{equation}\label{eq:skoda-integrability}
\int_U|f|^2e^{-\varphi_\nu}|g|^{-2(t+1+\varepsilon)}\,dV<\infty.
\end{equation}
Set
\[
 \alpha=\frac{t+\varepsilon}{p}>1.
\]
Then~\eqref{eq:skoda-integrability} is precisely the integrability hypothesis of the local analytic Skoda division theorem \cite{Sko72}. Hence there exist holomorphic functions
$h_1,\ldots,h_q$ such that $f=\sum_{j=1}^qg_jh_j$ and
\[
\int_U\left(\sum_{j=1}^q|h_j|^2\right)e^{-\varphi_\nu}|g|^{-2(t+\varepsilon)}\,dV<\infty.
\]
Thus
\[
h_j\in\mathcal I\!\left(\varphi_\nu+(t+\varepsilon)\log|g|^2\right)\subset\mathcal I\!\left(\varphi_\nu+t\log|g|^2\right),
\]
where the last inclusion uses $|g|^2\le1$. Therefore
\[
f\in\mathfrak a\,\mathcal I\!\left(\varphi_\nu+t\log|g|^2\right).
\]

It remains to consider $q=1$. Then $p=0$ and $\mathfrak a=(g)$. Let
$t\ge0$ and
\[
f\in\mathcal I\!\left(\varphi_\nu+(t+1)\log|g|^2\right).
\]
Since $\varphi_\nu$ is locally bounded above, after shrinking $U$ there is $C>0$ such that $e^{-\varphi_\nu}\ge C$.  Hence near $\{g=0\}$,
\[
\int_U\frac{|f|^2}{|g|^2}\,dV\le C^{-1}\int_U|f|^2e^{-\varphi_\nu}|g|^{-2(t+1)}\,dV<\infty.
\]
Thus $f/g$, initially holomorphic outside $\{g=0\}$, is locally square-integrable and extends holomorphically. Writing $f=gh$, one has
\[
|h|^2e^{-\varphi_\nu}|g|^{-2t}=|f|^2e^{-\varphi_\nu}|g|^{-2(t+1)},
\]
so
\[
h\in\mathcal I\!\left(\varphi_\nu+t\log|g|^2\right).
\]
This proves~\eqref{eq:scalar-skoda-real} in all cases.

Now let $k\ge k_0$. Then $kd\ge p$, and
\[
 (k+s)d=kd+r.
\]
Because locally bounded perturbations do not affect multiplier ideals,
\[
\mathcal I(\varphi_\nu+(k+s)\pi^*\psi)=\mathcal I\!\left(\varphi_\nu+(kd+r)\log|g|^2\right).
\]
Iterating~\eqref{eq:scalar-skoda-real} $r$ times gives
\[
\mathcal I\!\left(\varphi_\nu+(kd+r)\log|g|^2\right)=\mathfrak a^r\,\mathcal I\!\left(\varphi_\nu+kd\log|g|^2\right),
\]
which is~\eqref{eq:scalar-skoda-rational}. Tensoring with the free modules $\mathcal O_U^{m_\nu}$ and taking the direct sum over $\nu$ gives~\eqref{eq:J-rational}.
\end{proof}

Let
\[
 \mathcal V=\mathcal O_U(\pi^*E),\qquad\mathcal P=\bigoplus_{\nu=1}^N\mathcal O_U^{m_\nu}
 \qquad\text{and}\qquad\cA=\bigoplus_{\nu=1}^N A_\nu:\mathcal V\longrightarrow\mathcal P.
\]
The morphism $\cA$ is injective. Indeed, if $\cA(F)=0$, then the upper bound in the scalarization inequality gives $|F|_{\pi^*h}=0$ almost everywhere; positive definiteness of $h$ almost everywhere and holomorphicity imply $F=0$. Put
\[
\cG=\operatorname{im}\cA\subset \cP\qquad\text{and}\qquad F_k\cG=\cG\cap\mathcal J_k.
\]
As above, the scalarization formula and the injectivity of $\cA$ give a natural isomorphism
\begin{equation}\label{eq:gram-identification-rational}
\mathcal E(\pi^*h_k)|_U\xrightarrow{\ \sim\ }F_k\cG,\qquad F\longmapsto \cA(F).
\end{equation}

\subsection{Rees-standardness}

For each residue class
\[
 \ell\in\{0,\ldots,s-1\},
\]
let $t_\ell$ be the smallest integer satisfying
\[
t_\ell\ge k_0,\qquad t_\ell\equiv\ell\pmod s.
\]
Define
\[
\mathcal J_q^{(\ell)}:=\mathcal J_{t_\ell+sq},\qquad F_q^{(\ell)}\cG:=F_{t_\ell+sq}\cG.
\]
By Proposition~\ref{prop:scalar-skoda-rational},
\begin{equation}\label{eq:residue-scalar}
\mathcal J_{q+1}^{(\ell)}=\mathfrak a^r\,\mathcal J_q^{(\ell)},\qquad q\ge0.
\end{equation}
Let
\[
\mathcal R(\mathfrak a^r)=\bigoplus_{q\ge0}\mathfrak a^{rq}\,T^q
\]
be the Rees algebra of $\mathfrak a^r$, and put
\[
\mathcal R_F^{(\ell)}(\cG;t_\ell)=\bigoplus_{q\ge0}F_{t_\ell+sq}\cG\,T^q.
\]
By~\eqref{eq:residue-scalar}, this is naturally a graded $\mathcal R(\mathfrak a^r)$-module.

\begin{definition}
The filtered Gram image is said to be \emph{eventually $(r,s)$-Rees-standard} with respect to $\mathfrak a$ if there exists an integer $k_1\ge k_0$ such that
\[
F_{k+s}\cG=\mathfrak a^r\,F_k\cG\qquad\text{for every integer $k\ge k_1$}.
\]
Equivalently, for every residue class $\ell$, there exists an integer $t'_{\ell}\ge k_{1}$, $t'_{\ell}\equiv\ell\pmod s$, such that
\[
\bigoplus_{q\ge0}F_{t'_\ell+sq}\cG\,T^q
\]
is generated in degree zero over $\mathcal R(\mathfrak a^r)$.
\end{definition}

The key point is that eventual standardness is automatic.

\begin{theorem}\label{thm:AR-rational}
After shrinking $U$ if necessary, for each $\ell\in\{0,\ldots,s-1\}$ there exists an integer $c_\ell\ge1$ such that
\begin{equation}\label{eq:AR-residue}
F_{t_\ell+sq}\cG=\mathfrak a^{r(q-c_\ell)}\,F_{t_\ell+sc_\ell}\cG\qquad\text{for every $q\ge c_\ell$}.
\end{equation}
In particular, if
\[
k_1=\max_{0\le\ell\le s-1}\{\,t_\ell+sc_\ell\,\},
\]
then
\begin{equation}\label{eq:AR-eventual}
F_{k+s}\cG=\mathfrak a^r\,F_k\cG\qquad\text{for every integer $k\ge k_1$}.
\end{equation}
\end{theorem}

\begin{proof}
Fix $\ell$. Iterating~\eqref{eq:residue-scalar} gives
\[
\mathcal J_{t_\ell+sq}=(\mathfrak a^r)^q\,\mathcal J_{t_\ell},\qquad q\ge0.
\]
Hence
\begin{equation}\label{eq:AR-intersection}
F_{t_\ell+sq}\cG=\cG\cap\mathcal J_{t_\ell+sq}=\cG\cap(\mathfrak a^r)^q\,\mathcal J_{t_\ell}=(\cG\cap\mathcal J_{t_\ell})\cap(\mathfrak a^r)^q\,\mathcal J_{t_\ell}.
\end{equation}
Apply the Artin--Rees lemma to the coherent inclusion
\[
\cG\cap\mathcal J_{t_\ell}\subset\mathcal J_{t_\ell}
\]
with respect to the coherent ideal $\mathfrak a^r$. After shrinking $U$, there exists $c_\ell\ge1$ such that, for every $q\ge c_\ell$,
\[
(\cG\cap\mathcal J_{t_\ell})\cap(\mathfrak a^r)^q\,\mathcal J_{t_\ell}=(\mathfrak a^r)^{q-c_\ell}\,\left((\cG\cap\mathcal J_{t_\ell})\cap(\mathfrak a^r)^{c_\ell}\,\mathcal J_{t_\ell}\right).
\]
Using~\eqref{eq:AR-intersection} at $q$ and $c_\ell$ gives~\eqref{eq:AR-residue}. Consequently,
\[
F_{k_\ell+s(q+1)}\cG=\mathfrak a^r\,F_{k_\ell+sq}\cG\qquad(q\ge c_\ell).
\]
Taking the maximum over the finitely many residue classes yields~\eqref{eq:AR-eventual}.
\end{proof}

\begin{theorem}[=\cref{t13}, Skoda periodicity upstairs]\label{thm:upstairs-rational}
For every integer $k\ge k_1$,
\begin{equation}\label{eq:upstairs-rational}
\mathcal E(\pi^*h_{k+s})|_U=\mathfrak a^r\,\mathcal E(\pi^*h_k)|_U.
\end{equation}
\end{theorem}

\begin{proof}
Under the isomorphism~\eqref{eq:gram-identification-rational}, \eqref{eq:upstairs-rational} is precisely~\eqref{eq:AR-eventual}.
\end{proof}

\begin{remark}
The Artin--Rees theorem gives no reason for the stabilization index $k_1$ to coincide with the scalar threshold $k_0$. The Tor obstruction below measures whether the $s$-step periodicity already holds from $k_0$.
\end{remark}

\subsection{The induced filtration and its Tor obstruction}

We first formulate the algebraic criterion in an abstract setting. Let $U$ be a complex manifold, let $\cP$ be a locally free $\mathcal O_U$-module, let $\cG\subset \cP$ be coherent, and let
\[
\mathcal J_0\supseteq\mathcal J_1\supseteq\mathcal J_2\supseteq\cdots
\]
be a decreasing filtration by coherent submodules. Let $\mathfrak c\subset\mathcal O_U$ be a coherent ideal, fix an integer $s\ge1$, and assume
\begin{equation}\label{eq:abstract-sstep}
\mathcal J_{k+s}=\mathfrak c\,\mathcal J_k,\qquad k\ge k_0.
\end{equation}
Set
\[
 F_k\cG=\cG\cap\mathcal J_k\qquad\text{and}\qquad \cL_k=\frac{\mathcal J_k}{\cG\cap\mathcal J_k}.
\]
By the second isomorphism theorem,
\[
 \cL_k\simeq\frac{\cG+\mathcal J_k}{\cG},
\]
so $\cL_k$ is a coherent submodule of $\cP/\cG$.

\begin{proposition}\label{prop:tor-rational}
For every $k$ there is a natural right-exact sequence
\begin{equation}\label{eq:tor-sequence-rational}
\operatorname{Tor}_1^{\mathcal O_U}(\mathcal O_U/\mathfrak c,\mathcal J_k)\longrightarrow\operatorname{Tor}_1^{\mathcal O_U}(\mathcal O_U/\mathfrak c,\cL_k)\longrightarrow\frac{\cG\cap\mathfrak c\,\mathcal J_k}{\mathfrak c\,(\cG\cap\mathcal J_k)}\longrightarrow0.
\end{equation}
In particular, if
\[
\operatorname{Tor}_1^{\mathcal O_U}(\mathcal O_U/\mathfrak c,\cL_k)=0,
\]
then
\[
\cG\cap\mathfrak c\,\mathcal J_k=\mathfrak c\,(\cG\cap\mathcal J_k).
\]
\end{proposition}

\begin{proof}
Tensor
\[
0\longrightarrow F_k\cG\longrightarrow\mathcal J_k\longrightarrow \cL_k\longrightarrow0
\]
with $\mathcal O_U/\mathfrak c$. The relevant part of the long exact Tor
sequence is
\[
\operatorname{Tor}_1(\mathcal O_U/\mathfrak c,\mathcal J_k)\longrightarrow\operatorname{Tor}_1(\mathcal O_U/\mathfrak c,\cL_k)\longrightarrow\frac{F_k\cG}{\mathfrak c\,F_k\cG}\longrightarrow\frac{\mathcal J_k}{\mathfrak c\,\mathcal J_k}.
\]
The kernel of the last map is
\[
\frac{F_k\cG\cap\mathfrak c\,\mathcal J_k}{\mathfrak c\,F_k\cG}=\frac{\cG\cap\mathfrak c\,\mathcal J_k}{\mathfrak c\,(\cG\cap\mathcal J_k)},
\]
which proves~\eqref{eq:tor-sequence-rational}.
\end{proof}

\begin{theorem}[Tor criterion for $s$-step Rees-standardness]
\label{thm:tor-rational}
Assume~\eqref{eq:abstract-sstep} and
\[
\operatorname{Tor}_1^{\mathcal O_U}(\mathcal O_U/\mathfrak c,\cL_k)=0,\qquad k\ge k_0.
\]
Then
\[
F_{k+s}\cG=\mathfrak c\,F_k\cG,\qquad k\ge k_0.
\]
Consequently, for each residue class $\ell$ and the corresponding $k_\ell\ge k_0$,
\begin{equation}\label{eq:residue-standard}
F_{k_\ell+sq}\cG=\mathfrak c^q\,F_{k_\ell}\cG,\qquad q\ge0.
\end{equation}
\end{theorem}

\begin{proof}
Using \eqref{eq:abstract-sstep} and \cref{prop:tor-rational},
\[
F_{k+s}\cG=\cG\cap\mathcal J_{k+s}=\cG\cap\mathfrak c\,\mathcal J_k=\mathfrak c\,(\cG\cap\mathcal J_k)=\mathfrak c\,F_k\cG.
\]
Iteration along each residue class modulo $s$ gives~\eqref{eq:residue-standard}.
\end{proof}

In the present analytic setting, we apply this theorem with $\mathfrak c=\mathfrak a^r$.

\subsection{Geometric sufficient conditions}

We record several convenient sufficient conditions for
\[
\operatorname{Tor}_1^{\mathcal O_U}(\mathcal O_U/\mathfrak a^r,\cL_k)=0.
\]

\paragraph{The locally principal case.}
Suppose $\mathfrak a=(f)$, where $f$ is a non-zero divisor in $\mathcal O_U$. Then $\mathfrak a^r=(f^r)$ and
\[
\operatorname{Tor}_1^{\mathcal O_U}(\mathcal O_U/(f^r),\cL_k)\simeq\ker(f^r:\cL_k\to \cL_k).
\]
Hence it is enough that multiplication by $f$ be injective on $\cL_k$. Since $\cL_k\subset \cP/\cG$, it is enough that $f$ be a non-zero divisor on $\cP/\cG$, equivalently
\[
f\notin\bigcup_{\mathfrak p\in\operatorname{Ass}(\cP/\cG)}\mathfrak p.
\]
Geometrically, this means that $V(f)$ contains no associated component, including no embedded associated component, of $\cP/\cG$. If $\cG$ is locally a direct summand of $\cP$, then $\cP/\cG$ is locally free, and any holomorphic function $f$ which is not identically zero on a connected component is a non-zero divisor on $\cP/\cG$ and on every $\cL_k\subset \cP/\cG$.

\paragraph{Complete intersections.}
Suppose
\[
 \mathfrak a=(f_1,\ldots,f_c)
\]
is generated by an $\mathcal O_U$-regular sequence and that the same sequence is $\cL_k$-regular. The Koszul complex gives
\begin{equation}\label{eq:tor-a}
\operatorname{Tor}_j^{\mathcal O_U}(\mathcal O_U/\mathfrak a,\cL_k)=0,\qquad j>0.
\end{equation}
Since $\mathfrak a$ is a complete intersection ideal,
\[
\mathfrak a^m/\mathfrak a^{m+1}\simeq\operatorname{Sym}^m_{\mathcal O_U/\mathfrak a}(\mathfrak a/\mathfrak a^2)
\]
is locally free over $\mathcal O_U/\mathfrak a$.  Hence~\eqref{eq:tor-a} implies
\[
\operatorname{Tor}_j^{\mathcal O_U}(\mathfrak a^m/\mathfrak a^{m+1},\cL_k)=0,\qquad j>0.
\]
Induction on $m$ using the exact sequences
\[
0\longrightarrow\mathfrak a^m/\mathfrak a^{m+1}\longrightarrow\mathcal O_U/\mathfrak a^{m+1}\longrightarrow\mathcal O_U/\mathfrak a^m\longrightarrow0
\]
therefore yields
\[
\operatorname{Tor}_j^{\mathcal O_U}(\mathcal O_U/\mathfrak a^r,\cL_k)=0,\qquad j>0.
\]
A standard geometric sufficient condition for the required $\cL_k$-regularity is that, near $V(\mathfrak a)$, $\cL_k$ be Cohen--Macaulay, $\mathfrak a$ define a local complete intersection of codimension $c$, and
\[
\operatorname{codim}_{\operatorname{Supp}\cL_k}\bigl(V(\mathfrak a)\cap\operatorname{Supp}\cL_k\bigr)=c.
\]

\paragraph{Local freeness.}
If $\cL_k$ is locally free near $V(\mathfrak a)$, then it is flat there. Away from $V(\mathfrak a)$ the sheaf $\mathcal O_U/\mathfrak a^r$ vanishes. Hence all positive Tor sheaves vanish.

Combining these observations with \cref{thm:upstairs-rational,thm:tor-rational} yields the following.

\begin{theorem}[Skoda periodicity upstairs without delay]
\label{thm:upstairs-rational-no-delay}
Let $k_1$ be the stabilization index from \cref{thm:AR-rational}. Assume that
\[
\operatorname{Tor}_1^{\mathcal O_U}(\mathcal O_U/\mathfrak a^r,\cL_k)=0\qquad\text{for every integer }k_0\le k<k_1.
\]
Then
\begin{equation}\label{eq:no-delay-rational}
\mathcal E(\pi^*h_{k+s})|_U=\mathfrak a^r\,\mathcal E(\pi^*h_k)|_U,\qquad k\ge k_0.
\end{equation}
The Tor condition holds, in particular, under the corresponding geometric hypotheses described above.
\end{theorem}

\begin{proof}
For $k_0\le k<k_1$, \cref{prop:tor-rational}, applied with $\mathfrak c=\mathfrak a^r$ and the scalar relation $\mathcal J_{k+s}=\mathfrak a^r\,\mathcal J_k$, gives
\[
 F_{k+s}\cG=\mathfrak a^r\,F_k\cG.
\]
For $k\ge k_1$ the same equality follows from \cref{thm:AR-rational}. The isomorphism \eqref{eq:gram-identification-rational} then gives \eqref{eq:no-delay-rational}.
\end{proof}

\subsection{Descent and the defect sheaf}

We now assume that the analytic singularities of $\psi$ are induced by a coherent ideal sheaf $\mathfrak b\subset\mathcal O_X$ with the same rational coefficient $d=r/s$. Fix $x\in X$ and, after shrinking around $x$, choose a relatively compact neighbourhood $W\Subset X$ on which
\[
\psi=\frac rs\log\!\left(\sum_{j=1}^{q}|b_j|^2\right)+O(1),\qquad\mathfrak b=(b_1,\ldots,b_q).
\]
Put $X'_W=\pi^{-1}(W)$, $\pi_W=\pi|_{X'_W}$, and $\mathfrak a=\mathfrak b\,\mathcal O_{X'_W}$. On each scalarization chart, the generators $g_j$ may therefore be chosen as the pull-backs $\pi_W^*b_j$.

After shrinking $W$ once more, we may assume that $\overline W$ is compact and contained in the neighbourhood under consideration. Since $\pi$ is proper, $\pi^{-1}(\overline W)$ is compact and can therefore be covered by finitely many scalarization charts. Let $k_0$ be a common scalar threshold on this finite cover and let $k_1\ge k_0$ be a common Artin--Rees stabilization index. By \cref{thm:upstairs-rational},
\begin{equation}\label{e32}
\mathcal E(\pi_W^*h_{k+s})=\mathfrak a^r\,\mathcal E(\pi_W^*h_k),\qquad k\ge k_1.
\end{equation}
If the finite Tor conditions in \cref{thm:upstairs-rational-no-delay} hold on this cover, then \eqref{e32} holds for every $k\ge k_0$. Set
\[
\mathcal F_k=K_{\pi^{-1}(W)}\otimes\mathcal E(\pi_W^*h_k).
\]
There is a natural inclusion
\begin{equation}\label{e33}
\mathfrak b^r\,\pi_{W*}\mathcal F_k\subseteq\pi_{W*}(\mathfrak a^r\,\mathcal F_k).
\end{equation}
Indeed, a local section of the left-hand side is a finite sum $\sum_\gamma c_\gamma s_\gamma$ with $c_\gamma\in\mathfrak b^r$. After pull-back, each coefficient lies in $\mathfrak a^r$, so the resulting section belongs to $\mathfrak a^r\,\mathcal F_k$.

\begin{theorem}[=\cref{t14}, Local Skoda periodicity downstairs]\label{t59}
With $W,k_0,k_1$ as above, for every integer $k\ge k_1$,
\begin{equation}\label{e34}
\mathfrak b^r\,\mathcal E(h_k)|_W\subseteq\mathcal E(h_{k+s})|_W.
\end{equation}
Equality holds if and only if
\[
\pi_{W*}(\mathfrak a^r\,\mathcal F_k)=\mathfrak b^r\,\pi_{W*}\mathcal F_k.
\]
More precisely, there is a natural isomorphism
\begin{equation}\label{e35}
K_W\otimes\frac{\mathcal E(h_{k+s})|_W}{\mathfrak b^r\,\mathcal E(h_k)|_W}\simeq\frac{\pi_{W*}(\mathfrak a^r\,\mathcal F_k)}{\mathfrak b^r\,\pi_{W*}\mathcal F_k}.
\end{equation}
We call the quotient on the right the $(r,s)$-descent defect sheaf on $W$. If the finite Tor hypotheses of \cref{thm:upstairs-rational-no-delay} hold on the chosen cover, then the conclusions above hold for every $k\ge k_0$.
\end{theorem}

\begin{proof}
Upstairs periodicity gives $\mathcal F_{k+s}=\mathfrak a^r\,\mathcal F_k$ on $X'_W$. By \cref{prop:birational-covariance}, applied to $\pi_W$,
\[
\pi_{W*}\mathcal F_k=K_W\otimes\mathcal E(h_k)|_W\qquad\text{and}\qquad\pi_{W*}(\mathfrak a^r\,\mathcal F_k)=K_W\otimes\mathcal E(h_{k+s})|_W.
\]
Substituting these identities into \eqref{e33} and tensoring by $K_X^{-1}|_W$ gives \eqref{e34}. Because $K_W$ is invertible, passing to the quotient gives \eqref{e35}; the equality criterion and the assertion from $k_0$ follow immediately.

The defect also admits a useful restriction-map interpretation. Applying $\pi_{W*}$ to
\[
0\longrightarrow\mathfrak a^r\,\mathcal F_k\longrightarrow\mathcal F_k\longrightarrow\mathcal F_k/\mathfrak a^r\,\mathcal F_k\longrightarrow0
\]
gives a restriction morphism
\[
\rho_{k,W}:\pi_{W*}\mathcal F_k\longrightarrow\pi_{W*}(\mathcal F_k/\mathfrak a^r\,\mathcal F_k)
\]
with kernel $\pi_{W*}(\mathfrak a^r\,\mathcal F_k)$. Since $\mathfrak b^r\,\pi_{W*}\mathcal F_k\subseteq\ker\rho_{k,W}$, it factors through
\[
\beta^{(r,s)}_{k,W}:\frac{\pi_{W*}\mathcal F_k}{\mathfrak b^r\,\pi_{W*}\mathcal F_k}\longrightarrow\pi_{W*}(\mathcal F_k/\mathfrak a^r\,\mathcal F_k),
\]
and
\[
\ker\beta^{(r,s)}_{k,W}=\frac{\pi_{W*}(\mathfrak a^r\,\mathcal F_k)}{\mathfrak b^r\,\pi_{W*}\mathcal F_k}.
\]
Thus equality in \eqref{e34} on $W$ is equivalent to injectivity of $\beta^{(r,s)}_{k,W}$. 
\end{proof}

If $X$ is compact, finitely many such neighbourhoods $W$ cover $X$. Taking the maximum of their stabilization indices gives a global $k_1$, and \cref{t59} then holds on all of $X$ for every $k\ge k_1$.

\section*{Acknowledgments}

The author would like to thank Prof.Jixiang Fu for helpful discussions and valuable suggestions. The counterexample in \cref{sec:failure} was first carried out by the ChatGPT Pro, and the author subsequently checked, refined, and wrote out the construction. The author was supported by the NSFC, Grant No.~12271275.


\bigskip

\noindent
\textsc{Jingcao Wu}\\
School of Mathematics\\
Shanghai University of Finance and Economics\\
Shanghai 200433, People's Republic of China\\
\href{mailto:wujincao@shufe.edu.cn}
     {\texttt{wujincao@shufe.edu.cn}}

\end{document}